\documentclass[12pt]{article}
\usepackage[british]{babel}

\usepackage{dsfont}
\usepackage{mathrsfs}
\usepackage{url}
\usepackage{mathtools}
\usepackage{chngcntr}

\usepackage{latexsym,epsfig,amssymb,amsmath,amsthm,color,url}
\usepackage[inline,shortlabels]{enumitem}
\usepackage{setspace}
\usepackage{verbatim}

\usepackage[utf8]{inputenc}
\usepackage[T1]{fontenc}
\usepackage{lmodern}

\usepackage[hidelinks,colorlinks=true,urlcolor=blue,linkcolor=black,citecolor=black]{hyperref}
\newcommand*{\doi}[1]{\href{http://dx.doi.org/\detokenize{#1}}{doi}}

\allowdisplaybreaks

\numberwithin{equation}{section}
\newtheorem{theorem}[equation]{Theorem}

\newtheorem{lemma}[equation]{Lemma}

\theoremstyle{definition}

\theoremstyle{remark}

\newcommand{\Prob}{\mathbf{P}}
\newcommand{\E}{\mathbb{E}}

\newcommand{\toppling}{\mathfrak{t}}
\newcommand{\Ex}{\mathbf{E}}

\newcommand{\1}{\mathds{1}}

\begin{document}

\title{Uniform threshold for fixation of the \\ stochastic sandpile model on the line}
\author{Moumanti Podder \and Leonardo T.\ Rolla}
\maketitle

\begin{abstract}
We consider the abelian stochastic sandpile model. In this model, a site is deemed unstable when it contains more than one particle. Each unstable site, independently, is toppled at rate 1, sending two of its particles to neighbouring sites chosen independently. We show that when the initial average density is less than 1/2, the system locally fixates almost surely. We achieve this bound by analysing the parity of the total number of times each site is visited by a large number of particles under the sandpile dynamics.
\end{abstract}

%\subjclass[2010]{}

%\keywords{abelian stochastic sandpile, absorbing state phase transition, critical density, local fixation, uniform bound}

\section{Introduction}\label{intro}
We consider the abelian stochastic sandpile model on the integer line and obtain a uniform lower bound on the critical density threshold for its absorbing state phase transition. The stochastic sandpile model is a continuous time particle system on a finite or infinite graph, where the number of particles at a site $x$ is denoted by $\eta(x)$. An integer $k$ is specified and a site $x$ is declared \emph{unstable} when $\eta(x) \geqslant k$. Each unstable site $x$ \emph{topples} at rate $1$, sending $k$ of its particles to neighbours chosen independently according to a specified probability distribution. Our work focuses on the case $k = 2$. This abelian variation of Manna’s model~\cite{Manna91} is frequently considered in the physics and mathematics literature. We let $0 < q < 1$ denote the probability that each particle at $x$ is sent to $x-1$, so that $1-q$ is the probability with which it is sent to $x+1$. It was shown in~\cite{RollaSidoravicius12} that the critical density for this dynamics is positive, bounded from below by $q(1-q)$. In this paper we improve the analysis of~\cite{RollaSidoravicius12} and obtain a uniform lower bound of $\frac{1}{2}$ for every $0 < q < 1$. 

In \S\ref{background} below, we discuss physical motivation, related models and existing literature on these topics. In \S\ref{model}, we describe the model in detail and state the main result, and in \S\ref{proof_method:sec}, we discuss the main ideas involved in the proof of the result. In \S\ref{toppling_properties:sec} we state combinatorial properties of the sandpile dynamics in its particle-wise construction. In \S\ref{half_toppling:sec} we describe the stabilization strategy that is at the core of our proof. In \S\ref{half_toppling:subsec}, we define the \emph{half-toppling} operation; in \S\ref{traps:subsec}, we describe the setting of traps where the particles will ultimately be settled; in \S\ref{success_fixation:subsec}, we show that a successful execution of the stabilization strategy will imply that the system is in absorbing phase. Finally, in \S\ref{main:sec}, we show that the stabilization strategy succeeds with positive probability. 

\subsection{Background}\label{background} Stochastic sandpile models are an important example within a wide class of driven-dissipative lattice systems that naturally evolve to a critical state. Such systems are attracted to a critical equilibrium distribution without being specifically equipped with a tuning parameter for which a phase transition is observed. Some of the related models that can be found in the literature are \emph{abelian networks}, including \emph{oil and water model} and \emph{abelian mobile agents}~\cite{BondLevine16,BondLevine16a,BondLevine16b} and \emph{activated random walks}~\cite{Shellef10, AmirGurel-Gurevich10, CabezasRollaSidoravicius14, StaufferTaggi18, RollaTournier18, Taggi16, Taggi19, AsselahSchapiraRolla19, DickmanRollaSidoravicius10, Rolla19}. The latter is a special case of a class introduced by F. Spitzer in the 1970s. %\footnote{sort references} \footnote{is this the best reference?!} We refer the reader to~\cite{Redig06} for an elaborate discussion on these models. 

All these models involve strong non-locality of correlations and dynamic long-range effects, and the theory of \emph{self-organized criticality} attempts to justify this behaviour~\cite{Dhar99, Dhar06, Hinrichsen00}. The original models of self-organized criticality were defined on finite lattices. Consider the dynamics defined in a finite box in $\mathbb{Z}^{d}$. Whenever a particle moves to the exterior of the box, it is eliminated from the system. This is termed as \emph{dissipation}. Under this rule, since the box is finite, the system reaches an absorbing state after a finite time. At this point, a new particle is added uniformly randomly to the box. This is termed as \emph{driving}. This could destabilize the chosen site and possibly the entire box. The dynamics will then continue until a stable configuration is reached again. In this setting, the chain reactions or avalanches provoked by a single toppling are believed to display power law distributions at equilibrium and motivate the conjecture that they feature a critical behaviour. These models are termed \emph{driving-dissipation models} (DDM). 

In order to explain how the operations of driving and dissipation drift spontaneously towards criticality, a conservative lattice gas model known as the \emph{fixed energy sandpile} (FES) has been studied in~\cite{MunozDickmanPastor-SatorrasVespignaniZapperi01}. Simulations indicate that the FES model undergoes a phase transition at a critical density $\zeta_{c}$. For energy densities below $\zeta_{c}$, the FES goes to an absorbing state, and for densities above $\zeta_{c}$, it is in the active phase. This behaviour is similar to that exhibited in DDM models, where particles being gradually introduced are accommodated for inside the box as long as the density inside it remains below $\zeta_{c}$, leading to increases in the density. On the other hand, when the density is above $\zeta_{c}$, the toppling activity as well as dissipation at the boundary continue, bringing down the density. %(see~\cite{sornette2006},~\cite{MeesterQuant05}). 

Many interacting particle systems, such as cellular automata and rotor networks, studied in statistical physics and combinatorics, have an underlying abelian property which guarantees that the order of the interactions has no effect on the final state of the system. Common techniques that have been used to study such systems include the least action principle, local-to-global principles, burning algorithm, transition monoids and critical groups. A generalization of rotor networks and abelian sandpiles, known as the \emph{height arrow model}, has been studied in~\cite{PriezzhevDharDharKrishnamurthy96, DartoisRossin04}. The version of stochastic sandpile models that we study in this paper is particularly influenced by the generalizations explored in~\cite{DiaconisFulton91} and~\cite{Eriksson96}. Yet another generalization of sandpiles is the oil and water model, where each edge of the graph is marked as an \emph{oil edge} or \emph{water edge}, and two types of particles, the oil ones and the water ones, are considered. A vertex topples only if sufficient numbers of \emph{both} types of particles are present, sending one oil particle along each outgoing oil edge and one water particle along each outgoing water edge. Stochastic versions of the oil and water model have been studied in~\cite{AlcarazPyatovRittenberg09} and~\cite{CandelleroGangulyHoffmanLevineothers17}.

We mention here that the \emph{deterministic sandpile model} has also been studied in the physics literature, forming a universality class of its own, exhibiting strong non-ergodic effects. The toppling procedure is unable to eliminate certain microscopic symmetries in the configuration by the time the system becomes unstable, due to the existence of several toppling invariants in this model. The stationary state in deterministic sandpile models is uniform over a special subset of configurations exhibiting combinatorial properties and a well-studied algebraic structure. In stochastic sandpile models, by contrast, the driving-dissipation operations are themselves random, leading to a set of coupled polynomial equations~\cite{SadhuDhar09}
%\footnote{use \textasciitilde\ before \textbackslash ref and \textbackslash cite, without parenthesis}
and so few rigorous results regarding phase transitions in these models are known.

As noted in~\cite{RollaSidoravicius12}, the main questions that are pursued in this field are the attempts to understand the critical behaviour, the scaling relations and critical exponents, and whether the critical density $\zeta_{c}$ is the same as the asymptotic limit $\zeta_{s}$ observed in finite DDM systems at equilibrium as the size of the box grows.

The critical behaviour of stochastic sandpiles seems to belong to the same universality class as the depinning of a linear elastic interface subject to random pinning potentials, roughly depicted by a nailed carpet being detached from the floor by an external force of critical intensity, where the rupture of each nail induces other ruptures nearby, giving rise to avalanches~\cite{Dickman02}.

We discuss a few more notable results. In~\cite{SidoraviciusTeixeira17}, the existence of an absorbing phase is established for different models, including stochastic sandpile and activated random walks, for any dimension $d$, along with quantitative results on expected time of absorption and warm-up phase of the DDM models. In~\cite{FeyMeester15}, a discrepancy between stationary density and transitive density is shown to exist in a ``quenched'' version of Manna's sandpile model. 

\subsection{Description of the stochastic sandpile model}
\label{model}
We describe the \emph{stochastic sandpile model} (abbreviated henceforth as SSM) on $\mathbb{Z}$. At time $t$, the \emph{configuration} of this model is given by the (random) tuple $\eta_{t} = (\eta_{t}(x): x \in \mathbb{Z}) \in (\mathbb{N}_{0})^{\mathbb{Z}}$, where $\eta_{t}(x)$ is the random variable indicating the number of particles at site $x$ at time $t$. We assume that the distribution $\nu$ of the initial configuration $\eta_{0}$ of the particles is given by i.i.d.\ random variables $\eta_{0}(x)$ for $x \in \mathbb{Z}$ such that $\E[\eta_{0}(0)] = \zeta$. We let $\mathbb{P}^{\nu}$ denote the law of the stochastic process $\{\eta_{t}: t \geqslant 0\}$. 

We call a site $x \in \mathbb{Z}$ \emph{stable} at time $t$ if $\eta_{t}(x) \leqslant 1$, otherwise \emph{unstable}. Each unstable site $x$ \emph{topples} according to a Poisson clock, independent of all other sites, sending $2$ of the particles at $x$ to neighbours $y$ and $z$ of $x$ chosen independently with transition kernel $p(-1) = q$ and $p(1) = 1-q$ for some fixed $q \in [0,1]$. Formally, we can describe this process as follows: given any configuration $\eta \in (\mathbb{N}_{0})^{\mathbb{Z}}$, for each $x \in \mathbb{Z}$, the transition $\eta \rightarrow \toppling_{x,y} \toppling_{x,z} \eta$ happens at rate $A(\eta(x)) p(y-x) p(z-x)$, where 
%\begin{equation}%\label{instructions}
\[
\toppling_{x,y}\eta(z) = 
 \begin{cases} 
 \eta(x)-1 & \text{if } z = x,\\
 \eta(y)+1 & \text{if } z = y,\\
 \eta(z) & \text{otherwise,}
 \end{cases}
\]
%\end{equation}
the indicator function $A(k) = \1_{k \geqslant 2}$ indicates whether the site is unstable or not, and $p(-1) = 1 - p(+1) = q$ and $p(z) = 0$ for all $z \neq \pm 1$. We say that the system \emph{locally fixates} if every site $x$ is visited only finitely many times by the particles and eventually becomes stable, i.e.\ $\eta_{t}(x)$ is eventually constant for each $x$; otherwise we say that the system \emph{stays active}. 

\begin{theorem}\label{main}
Consider the above-described dynamics of the stochastic sandpile model, with the distribution of the initial configuration $\eta_{0}$ given by i.i.d.\ random variables $\eta_{0}(x)$ for all $x \in \mathbb{Z}$ with $\E[\eta_{0}(0)] = \zeta$. If $\zeta < \frac{1}{2}$, then the system locally fixates almost surely.
\end{theorem}
 
\subsection{About the proof method}
\label{proof_method:sec}

We recall the argument used in~\cite{RollaSidoravicius12} to prove that $\zeta_c \geqslant \frac{1}{4}$ for $q = \frac{1}{2}$ (or $\zeta_c \geqslant q(1-q)$ in general) and comment on the new ideas introduced here to get a uniform bound $\zeta_c \geqslant \frac{1}{2}$ for all $0<q<1$.

The argument builds upon the \emph{site-wise representation}, which focuses on the total number of jumps rather than the order in which they occur.
Then one introduces the \emph{half-toppling} operation whereby stable particles that would otherwise remain still can be forced to move.
This new framework is related to the original representation for an appropriate definition of stability: sites which have been half-toppled an odd number of times are deemed unstable if they contain at least one particle.
A key property is that half-toppling stable sites cannot decrease the amount of activity in the system.

With this tool at hand, a typical tool to study fixation vs.\ activity is to specify a \emph{toppling procedure}: an explicit choice of which site should be toppled next, possibly after a preliminary stage where some bits of the underlying randomness are examined.
A toppling procedure is good if its analysis can yield good estimates.
In~\cite{RollaSidoravicius12},
an explorer is launched from each particle's position and finds a suitable \emph{trap} for that particle.
To keep the outcome of each explorer independent of previous ones, the procedure only uses properties intrinsic to the current exploration when setting up the trap, in a way that neither depends on the randomness exposed in previous steps, nor reveals any information which could affect the conditional distribution of subsequent steps.

The novelty here is that, instead of keeping the randomness completely hermetic between different steps, on the contrary we consider the parity at each site resulting from the passage of \emph{all previous explorers}.
We observe that the field of parities becomes roughly independent, with probabilities approximately $ \frac{1}{2} $.
This way, the trap can be set up using properties of the current explorer's trajectory combined with parity considerations.

\section{Site-wise representation}\label{toppling_properties:sec}

In order to define the random topplings, we start with a set of mutually independent random instructions $\mathscr{I} = (\toppling^{x,j}: x \in \mathbb{Z}, j \in \mathbb{N})$ where, for each $j \in \mathbb{N}$, the random variable $\toppling^{x,j}$ equals $\toppling_{x,x-1}$ with probability $q$ and $\toppling_{x,x+1}$ with probability $1-q$. We make $\mathscr{I}$ independent of the initial configuration $\eta_{0}$, and we let $\Prob^{\nu}$ denote the joint law of $(\eta_{0}, \mathscr{I})$. The coordinate $h(x)$ of the \emph{odometer} tuple $h = (h(x): x \in \mathbb{Z})$ denotes the number of topplings that have occurred at $x$ \emph{so far}. The toppling operation at a site $x$ that is unstable under a given configuration $\eta$, is defined by $\Phi_{x}(\eta, h) = (\toppling^{x, 2h(x)+2} \toppling^{x, 2h(x)+1} \eta, h + \delta_{x})$, where 
\[
(h + \delta_{x})(y) = 
\begin{cases} 
 h(x) + 1 & \text{if } y = x,\\
 h(y) & \text{otherwise.}
\end{cases}
\]
We let $h = \mathbf{0}$ indicate $h(y) = 0$ for all $y \in \mathbb{Z}$. We abbreviate $\Phi_{x}(\eta, \mathbf{0})$ as $\Phi_{x} \eta$. The toppling operation $\Phi_{x}$ at $x$ is termed \emph{legal} for a configuration $\eta$ when $x$ is unstable, i.e.\ when $\eta(x) \geqslant 2$, and \emph{illegal} otherwise. In the actual dynamics, illegal topplings do not occur.

In the rest of this section, the toppling instruction $\toppling^{x, j}$ is fixed and known for every site $x \in \mathbb{Z}$ and every $j \in \mathbb{N}$, and no randomness is involved. Given a sequence $\alpha = (x_{1}, x_{2}, \ldots, x_{k})$ of sites, we denote by $\Phi_{\alpha} = \Phi_{x_{k}} \Phi_{x_{k-1}} \cdots \Phi_{x_{2}} \Phi_{x_{1}}$ the composition of the topplings at the sites $x_{1}, x_{2}, \ldots, x_{k-1}, x_{k}$, in that order. We call $\Phi$ \emph{legal} for a configuration $\eta$ if $\Phi_{x_{1}}$ is legal for $\eta$ and $\Phi_{x_{\ell}}$ is legal for $\Phi_{(x_{1}, \ldots, x_{\ell-1})} \eta$ for each $\ell = 2, \ldots, k$. In this case we call $\alpha$ a \emph{legal sequence} of topplings for $\eta$. We define
\begin{equation}\label{m_defn}
m_{\alpha} = (m_{\alpha}(x): x \in \mathbb{Z}), \quad \text{where} \quad m_{\alpha}(x) = \sum_{\ell=1}^{k} \1_{x_{\ell} = x}
\end{equation}
to be the field that indicates the number of times a site $x$ is toppled in the sequence $\alpha$, for each $x \in \mathbb{Z}$. Henceforth, given any two fields $\gamma = (\gamma_{x}: x \in \mathbb{Z})$ and $\gamma' = (\gamma'_{x}: x \in \mathbb{Z})$, we define $\gamma \geqslant \gamma'$ if, for every $x \in \mathbb{Z}$, we have $\gamma(x) \geqslant \gamma'(x)$. 

Given a finite subset $V$ of $\mathbb{Z}$ and a configuration $\eta$, we say that $\eta$ is \emph{stable inside} $V$ if every site $x \in V$ is stable in $\eta$, i.e.\ $\eta(x) \leqslant 1$ for each $x \in V$. Given a toppling sequence $\alpha = (x_{1}, x_{2}, \ldots, x_{k})$, we say that $\alpha$ is \emph{contained in $V$} if $x_{\ell} \in V$ for all $\ell = 1, \ldots, k$. We say that $\alpha$ \emph{stabilizes} a given configuration $\eta$ in $V$ if every site $x \in V$ is stable in $\Phi_{\alpha}\eta$. 

We refer to [\cite{RollaSidoravicius12}, Lemmas 1, 2 and 3] for the following lemmas:
\begin{lemma}\label{least_action_principle}
If $\alpha$ and $\beta$ are both legal sequences of topplings for a configuration $\eta$, the sequence $\beta$ is contained in a finite subset $V$ of $\mathbb{Z}$ and $\alpha$ stabilizes $\eta$ in $V$, then $m_{\beta} \leqslant m_{\alpha}$.
\end{lemma}

\begin{lemma}\label{abelian_toppling}
Given a finite subset $V$ of $\mathbb{Z}$, a configuration $\eta$, and two toppling sequences $\alpha$ and $\beta$ that are legal for $\eta$, contained in $V$ and stabilizing $\eta$ in $V$, we have $m_{\alpha} = m_{\beta}$.
\end{lemma}
This lemma allows us to define $m_{V, \eta} = m_{\alpha}$ for any toppling sequence $\alpha$ that is legal for $\eta$, contained in $V$ and stabilizes $\eta$ in $V$.

\begin{lemma}\label{monotonic_toppling}
If $V$ and $V'$ are two finite subsets of $\mathbb{Z}$ with $V \subseteq V'$, and $\eta$ and $\eta'$ are two configurations with $\eta \leqslant \eta'$, then $m_{V, \eta} \leqslant m_{V', \eta'}$.
\end{lemma}

Due to this monotonicity result, choosing \emph{any} sequence $\{V_{n}\}_{n \in \mathbb{N}}$ of monotonically increasing finite subsets of $\mathbb{Z}$, i.e.\ $V_{n} \uparrow \mathbb{Z}$, we can define the limit $m_{\eta} = \lim_{n \rightarrow \infty} m_{V_{n}, \eta}$, and this limit does not depend on the particular sequence of subsets chosen. The following result, used in this paper to establish fixation when $\zeta < \frac{1}{2}$, is Lemma 4 of~\cite{RollaSidoravicius12}. Recall $\nu$, $\mathscr{I}$, $\mathbb{P}^{\nu}$ and $\Prob^{\nu}$ from discussions above. 
\begin{lemma}\label{m(0)_finite}
Let $\nu$ be a translation-invariant, ergodic distribution with finite density $\nu(\eta_{0}(0))$. Then 
\begin{equation}
\mathbb{P}^{\nu}(\text{the system locally fixates}) = \Prob^{\nu}(m_{\eta_{0}}(0) < \infty) \in \{0,1\}. \nonumber
\end{equation}
\end{lemma}
This theorem tells us that, in order to establish local fixation almost surely, it is enough to show that the origin is toppled a finite number of times with positive probability.

%$\clubsuit$
%FROM THIS POINT ONWARDS WE ARE SWITCHING FROM $\mathbb{P}$ TO $\Prob$, so we should also be using $\mathbf{E}$ instead of $\E$.
%$\clubsuit$

\section{Half-toppling and strategy for stabilization}\label{half_toppling:sec}
At the very outset of this section, we mention that much of the description of the stabilization algorithm is the same as that in~\cite{RollaSidoravicius12}. We construct an algorithm that, given any random pair $(\eta_{0}, \mathscr{I})$ generated according to the aforementioned distribution, tries to stabilize the particles according to the instructions in $\mathscr{I}$, but includes some additional, ``semi-legal'' single-topplings, whereby stable sites, which otherwise would have remain untouched, are now \emph{half-toppled}. We need to make sure that whenever the algorithm is \emph{successful}, we have $m_{\eta}(0) = 0$, and that it succeeds with positive probability. 

This algorithm prescribes a settling procedure or rule for each particle, by defining a suitable \emph{trap} for the particle. We imagine an \emph{explorer} associated with each particle, and the trap for the particle is decided by the trajectory of this explorer. It may be helpful for the reader to visualize the explorer as acting like a scout for an army, riding ahead and inspecting the terrain before the army moves. The exploration follows the path that the particle would traverse if we always toppled the site it occupies, and it stops at the site where the trap has been placed for the particle. We declare the algorithm \emph{successful} if we are able to define a suitable trap for every particle. The trap will always lie on the exploration path of the particle, but it need not lie at the end of the exploration path, as we generally have to expose some instructions in $\mathscr{I}$ further along the exploration path before we can decide where the trap must lie. Once the trap has been chosen, the particle is moved along the exploration path until it reaches the trap, where it is settled.

We make sure that the following conditions are satisfied by this algorithm. We do not disturb the particles that have already been settled, which requires that a particle is not allowed to visit the sites where previous particles have been settled already. We try to ensure that the traps for different particles are placed as close to each other as possible, so that the next particles have greater room for exploration. We call an instruction $\toppling^{x,j}$ in $\mathscr{I}$ \emph{corrupted} (and the corresponding site $x$ \emph{corrupted} as well) if this instruction has been exposed while deciding where to place the trap for some particle but is not actually executed by the particle. We make sure that the corrupted site $x$ is not visited by the subsequent particles, so as to ensure mutual independence of explorations of different particles. We also try to keep the regions containing the corrupted sites as compact as possible. 

\subsection{Half-topplings}\label{half_toppling:subsec} In this subsection, we define \emph{half-topplings}, relaxing the restrictions of the original dynamics, and discuss their properties in the spirit of \S\ref{toppling_properties:sec}. In a half-toppling, a single particle from the site concerned, say $x$, is sent to the neighbouring site $x-1$ with probability $q$ and to $x+1$ with probability $1-q$. If for some $x \in \mathbb{Z}$ and some configuration $\eta$, we have $\eta(x) = 1$, then $x$ is stable if it has been half-toppled an even number of times, and unstable if it has been half-toppled an odd number of times. A half-toppling at $x$ under the configuration $\eta$ is considered legal if $x$ is unstable, and \emph{semi-legal} when $\eta(x) \geqslant 1$, regardless of whether $x$ is stable or not. 

The half-toppling operation performed on the configuration $\eta$, at the site $x$, is denoted $\varphi_{x} \eta$, so that given an odometer $h$, we define $\varphi_{x}(\eta, h) = (\toppling^{x, 2h+1} \eta, h + \frac{1}{2} \delta_{x})$. As before, we abbreviate $\varphi_{x}(\eta, \mathbf{0})$ by $\varphi_{x} \eta$. This definition implies that $\Phi_{x} = \varphi_{x} \varphi_{x}$, i.e.\ the composition of two half-topplings at $x$ constitutes a toppling at $x$.

Definitions analogous to those introduced for the toppling operations can also be introduced for half-topplings. Given a sequence $\beta = (y_{1}, y_{2}, \ldots, y_{k})$ of sites, we define the sequence of half-topplings $\varphi_{\beta}$ as the composition $\varphi_{y_{k}} \varphi_{y_{k-1}} \cdots \varphi_{1}$ of half-topplings. We define, for any site $y \in \mathbb{Z}$,
\begin{equation}\label{widetilde_m_defn}
\widetilde{m}_{\beta}(y) = \sum_{\ell=1}^{k} \1_{y_{\ell} = y}, \quad \text{and} \quad \widetilde{m}_{\beta} = (\widetilde{m}_{\beta}(y): y \in \mathbb{Z}).
\end{equation}
We note here that although there is no difference between the definitions in \eqref{m_defn} and \eqref{widetilde_m_defn}, the $\widetilde{m}_{\beta}$ emphasizes that the $\beta$ here refers to a sequence of half-topplings and not a sequence of topplings. Given such a $\beta$, we define $\beta^{2} =(y_{1}, y_{1}, y_{2}, y_{2}, \ldots, y_{k}, y_{k})$, which gives us $m_{\beta} = \frac{1}{2}\widetilde{m}_{\beta^{2}}$. As before, given $V$ and $\beta = (y_{1}, y_{2}, \ldots, y_{k})$, we say that $\beta$ is \emph{contained in $V$} if $x_{\ell} \in V$ for all $\ell = 1, \ldots, k$. We say that $\beta$ \emph{stabilizes} $\eta$ in $V$ if every $x \in V$ is stable in $\varphi_{\beta}\eta$, i.e.\ for each site $x \in V$, either $\varphi_{\beta}\eta(x) = 0$ or $\varphi_{\beta}\eta(x) = 1$ and the site has been half-toppled an even number of times. We refer to [\cite{RollaSidoravicius12}, Lemma 6] for the following lemma.

\begin{lemma}\label{least_action_half_toppling}
Let $\beta$ be a sequence of half-topplings contained in $V$ and legal for $\eta$, and $\alpha$ another sequence of half-topplings that is semi-legal for $\eta$ and stabilizes $\eta$ in $V$. Then $\widetilde{m}_{\beta} \leqslant \widetilde{m}_{\alpha}$.
\end{lemma}

\subsection{Procedure for setting traps}\label{traps:subsec} Recall the random walk bias parameter $q$ from \S\ref{model}. By symmetry we may assume that $q \geqslant \frac{1}{2}$. We shall focus on the positive half-line of integers $\mathbb{Z}^{+}$. The argument for the particles on $\mathbb{Z}^{-}$ is symmetric when $q = \frac{1}{2}$, and easier when $q > \frac{1}{2}$, as discussed farther below.

Let $0 < x_{1} < x_{2} < x_{3} < \cdots$ denote the positions of the particles under the initial configuration $\eta_{0}$. Given realizations of $\mathscr{I}$ and $\eta_{0}$, we construct the settling procedure as follows. We set $b_{0} = 0$, and consider the first particle, located at $x_{1}$. The explorer associated with this particle now reveals the instructions in $\mathscr{I}$ one by one, starting at $x_{1}$ and then following the path that would have been traversed by the particle if we were to half-topple the site currently containing it, over and over again, until it hits barrier $b_{0}$ for the first time. 

We now set the trap for the first particle. We let $\tau_{1}$ denote the first time that it hits barrier $b_{0}$, and let $\{X^{1}_{n}\}_{0 \leqslant n \leqslant \tau_{1}}$ denote the trajectory of the first explorer, where $X^{1}_{0} = x_{1}$. We define
\begin{equation}%\label{a_{1}_defn}
a_{1} = \min\Big\{x \in \mathbb{Z}: \sum_{n=0}^{\tau_{1}} \1_{X^{1}_{n} = x} \equiv 1 \bmod 2\Big\}. \nonumber
\end{equation}
Let 
\begin{equation}
\theta_{1} = \max\big\{1 \leqslant n < \tau_{1}: X^{1}_{n} = X^{1}_{n-1} + 1\big\}, \nonumber
\end{equation}
so that the last time that the first explorer jumps to the right is at time $\theta_{1}-1$. We then let 
\begin{equation}%\label{c_{1}_defn}
c_{1} = X^{1}_{\theta_{1}}. \nonumber
\end{equation}
We then choose
\begin{equation}%\label{b_{1}_defn}
b_{1} = \min\{a_{1}, c_{1}\}, \nonumber
\end{equation}
which will serve as the barrier for the second explorer starting at $x_{2}$. The trap for the first particle has now been set up at $b_{1}$. If both $a_{1}$ and $c_{1}$ are greater than $x_{1}$, we define $b_{k} = \infty$ for every $k \in \mathbb{N}$, stop and declare the settling procedure to have failed. If the trap for the first particle is successfully set up, we have $b_{1} \in [b_{0}+1, x_{1}] \cap \mathbb{Z}$. 

Suppose the first $k$ traps have been laid down already, at positions $0 < b_{1} < b_{2} < \cdots < b_{k}$. The $(k+1)$-st explorer, starting at $x_{k+1}$, reveals and follows, one by one, the instructions from $\mathscr{I}$ corresponding to the current site it is occupying, that have not been revealed by the previous $k$ explorers, until it hits $b_{k}$ for the first time. Its exploration is described by $\{X^{k+1}_{n}\}_{0 \leqslant n \leqslant \tau_{k+1}}$, where $X^{k+1}_{0} = x_{k+1}$ and $\tau_{k+1}$ is the first time that it hits $b_{k}$. 

We define
\begin{equation}\label{a_{k}_defn}
a_{k+1} = \min\Big\{x > b_{k}: \sum_{j=1}^{k+1} \sum_{n=0}^{\tau_{j}} \1_{X^{j}_{n} = x} \equiv 1 \bmod 2\Big\}.
\end{equation} 
Once again, we define
\begin{equation}
\theta_{k+1} = \max\big\{1 \leqslant n \leqslant \tau_{k+1}: X^{k+1}_{n} = X^{k+1}_{n-1}+1\big\}
\end{equation}
so that $\theta_{k+1}-1$ is the last time that the $(k+1)$-st explorer jumps to the right. Then we set
\begin{equation}\label{c_{k}_defn}
c_{k+1} = X^{k+1}_{\theta_{k+1}},
\end{equation}
and finally define the trap for the $(k+1)$-st particle to be at the site
\begin{equation}\label{b_{k}_defn}
b_{k+1} = \min\{a_{k+1}, c_{k+1}\}.
\end{equation}
Once again, if both $a_{k+1}$ and $c_{k+1}$ exceed $x_{k+1}$, we stop and declare the settling procedure to have failed, and define $b_{j} = \infty$ for all $j \geqslant k+1$. If the trap can be set up successfully, we have $b_{k+1} \in [b_{k}+1, x_{k+1}] \cap \mathbb{Z}$. As mentioned earlier, we declare the entire settling procedure to have been successful when it is possible to define a finite $b_{k}$ for every $k \in \mathbb{N}$.

\subsection{Success implies fixation}\label{success_fixation:subsec} We show here that, for given realizations of $\mathscr{I}$ and $\eta_{0}$, if the settling procedure succeeds, then $m_{\eta_{0}}(0) = 0$, which in turn implies fixation. To this end, we make use of Lemma~\ref{m(0)_finite} and the same argument as that given in~\cite{RollaSidoravicius12}, showing that if the settling procedure succeeds for all particles on both $\mathbb{Z}^{+}$ and $\mathbb{Z}^{-}$, then every finite interval can be stabilized with semi-legal half-topplings without ever toppling the origin. We note here that each step of an explorer corresponds either to a semi-legal half-toppling of a stable site, or a legal half-toppling of an unstable site, and if a site is visited more than once, then at least one of the last two moves has to correspond to the former scenario.

Suppose we have settled the particles starting at $x_{1}$, $x_{2}$, $\ldots$, $x_{k}$ and established that, at this point, the sites $b_{1}$, $b_{2}$, $\ldots$, $b_{k}$ are stable. If $b_{k+1} = c_{k+1}$, then the sites found by the $(k+1)$-st explorer have been successively half-toppled until the second-last half-toppling at $c_{k+1}$. If $c_{k+1}$ has been half-toppled an even number of times so far, then it is stable, and the particle is left there without executing any further half-toppling instruction revealed by the $(k+1)$-st explorer. If $c_{k+1}$ has been half-toppled an odd number of times so far, then the second-last half-toppling at $c_{k+1}$ is executed and the particle is again left at $c_{k+1}$. This ensures that $c_{k+1}$ is stable. After the last two instructions at $c_{k+1}$ had been revealed by the $(k+1)$-st explorer, it did not jump to the right again. Thus all half-toppling instructions revealed by the $(k+1)$-st explorer, except possibly some that lie in the interval $[b_{k}+1, c_{k+1}] \cap \mathbb{Z}$, are executed. Thus the instructions and sites corrupted by the $(k+1)$-st explorer lie inside $[b_{k}+1, c_{k+1}] \cap \mathbb{Z}$ and are never encountered by the $(k+2)$-nd explorer.

Now we consider the scenario where $b_{k+1} = a_{k+1} < c_{k+1}$. Note that the particle starting at $x_{k+1}$ is left at the site $a_{k+1}$ without executing the last half-toppling instruction at $a_{k+1}$ that the $(k+1)$-st explorer reveals. By definition of $a_{k+1}$, the total number of times $a_{k+1}$ is visited by the explorers starting at $x_{1}$, $x_{2}$, $\ldots$, $x_{k+1}$ is odd. This means that the total number of times the site $a_{k+1}$ has been half-toppled, not counting the last instruction at $a_{k+1}$ revealed by the $(k+1)$-st explorer, is even. Therefore, once the particle starting at $x_{k+1}$ is settled at $a_{k+1}$, the site $a_{k+1}$ is stabilized. The sites corrupted by the $(k+1)$-st explorer all lie in $[b_{k}+1, a_{k+1}] \cap \mathbb{Z}$, and are never encountered by the $(k+2)$-nd explorer during its journey.

Thus, by induction, we conclude that by the time the $(k+1)$-st particle is settled, all sites in $[0, b_{k+1}] \cap \mathbb{Z}$ are stable and the origin remains untouched. This argument works by symmetry for particles on $\mathbb{Z}^{-}$ when $q = \frac{1}{2}$. In the case where $q > \frac{1}{2}$, the event that each particle on $\mathbb{Z}^{-}$ escapes to $-\infty$ without ever hitting the origin has strictly positive probability. Therefore, on that event, any sequence of semi-legal topplings on $\mathbb{Z}^{-}$ keeps the origin untouched. This completes the argument that a successful settling procedure implies fixation.

In \S\ref{main:sec}, we show that the settling procedure succeeds with positive probability, implying that $\Prob^{\nu}(m_{\eta_{0}}(0) = 0) > 0$. Combined with Lemma~\ref{m(0)_finite}, this establishes that the system locally fixates almost surely. 

\section{Positive probability of success}\label{main:sec}
Suppose the initial average density $\zeta < \frac{1}{2}$ is fixed. Fix $\gamma_{1}$ and $\gamma_{2}$ such that $2 < \gamma_{2} < \gamma_{1} < \frac{1}{\zeta}$. By the strong law of large numbers, $\frac{k}{x_{k}} \rightarrow \zeta$ almost surely. Moreover, the event $\mathcal{A} = \{x_{k} > \gamma_{1} k \text{ for all } k \in \mathbb{N}\}$ has positive probability. 
%\begin{equation}\label{gamma_parameters}
%2 < \gamma_{2} < \gamma_{1} < \frac{1}{\zeta}.
%\end{equation}
%Let $\mathcal{A}$ denote the $\sigma$-field containing the values of $x_{k}$ for all $k \in \mathbb{N}$. 

We define the $\sigma$-field $\mathcal{G}_{k}$ containing the values of $b_{1}, b_{2}, \ldots, b_{k}$ as well as the starting site $x_{i}$ for each $i \in \mathbb{N}$. Given positive integers $M$ and $N$, we define the positive integer
\begin{equation}\label{K_{M,N}}
K_{M,N} = \Big\lceil \frac{\gamma_{1} M + N}{\gamma_{1} - \gamma_{2}} \Big\rceil,
\end{equation}
and for each $k \geqslant K_{M,N}$, we define the event
%\begin{multline}
%\mathcal{B}_{k}^{M, N} = \{b_{j} < \gamma_{2} j \text{ for all } j = 1, \ldots, K_{M,N},\ b_{j} < b_{K_{M,N}} + \gamma_{2} (j-K_{M,N}) \\ \text{ for all } j = K_{M,N}+1, \ldots, k\}
%\end{multline}
%In particular, $M$ and $N$ will depend on $\gamma_{1}$, $\gamma_{2}$ and $q$, where $q$ is the random walk parameter introduced in \S\ref{model}.
\begin{multline}\label{B_{k}^{M,N}_event}
\mathcal{B}_{k}^{M, N} = \{b_{K_{M,N}} < \gamma_{2} K_{M,N} \text{ and } b_{K_{M,N}+j} \leqslant b_{K_{M,N}} + \gamma_{2} j \text{ for all }\\ j = 1, \ldots, k - K_{M,N}\}, \nonumber
\end{multline}
which has positive probability. The purpose of choosing $K_{M,N}$ as in \eqref{K_{M,N}} is to guarantee that on the event $\mathcal{B}_{k}^{M,N} \cap \mathcal{A}$, we have $b_{k}+N < x_{k-M}$, so that it leaves a large enough terrain for the estimates we shall make for the $(k+1)$-st explorer. 

The event $\mathcal{B}_{k}^{M, N}$ also requires that the settling procedure has not failed for the first $k$ particles on $\mathbb{Z}^{+}$. The sequence $\{\mathcal{B}_{k}^{M,N}\}_{k \in \mathbb{N}}$ is decreasing and we denote its limit by $\mathcal{B}_{\infty}^{M,N}$. That the settling procedure succeeds with positive probability follows from the following lemma. 
\begin{lemma}\label{main_2}
We can choose $M$ and $N$ so that $\Prob[\mathcal{B}_{\infty}^{M, N}] > 0$.
\end{lemma}
The proof of Lemma~\ref{main_2} is accomplished based on the following lemma.
\begin{lemma}\label{main_lemma}
Given $\varepsilon > 0$, there exists a positive integer $M = M_{\varepsilon}$ such that for all $N \geqslant 2$ and all $k \geqslant K_{M,N}$, on the event $\mathcal{B}_{k}^{M, N} \cap \mathcal{A}$, 
\begin{equation}\label{b_{k+1}-b_{k}_tail_upto_N}
\Prob[a_{k+1} - b_{k} > s\big|\mathcal{G}_{k}] \leqslant \Big(\frac{1}{2}+\varepsilon\Big)^{s} \quad \text{for all } s \leqslant N,
\end{equation}
and
\begin{equation}\label{b_{k+1}-b_{k}_tail_beyond_N}
\Prob[c_{k+1} - b_{k} > s\big|\mathcal{G}_{k}] \leqslant \rho^{s} \quad \text{for all } s > N,
\end{equation}
for some constant $\rho \in (0,1)$.
\end{lemma}
As a side remark, the constant $\rho$ depends only on $q$, but we do not need this fact in the proof of Lemma~\ref{main_lemma}. The proof of Lemma~\ref{main_lemma} involves several steps. In \S\ref{proof:subsec:1}, the trajectories of the explorers are identified with suitable $h$-transforms of random walks, in \S\ref{proof:subsec:2}, a general result regarding the parity of the sum of i.i.d.\ integer-valued random variables is discussed, and \S\ref{proof:subsec:3} and \S\ref{proof:subsec:4} are respectively dedicated to the proofs of \eqref{b_{k+1}-b_{k}_tail_upto_N} and \eqref{b_{k+1}-b_{k}_tail_beyond_N}. 

We now prove that Lemma~\ref{main_lemma} implies Lemma~\ref{main_2}.
%\begin{equation}\label{varepsilon_choice}
%\varepsilon < \frac{\gamma_{2} - 2}{2(\gamma_{2}+2)},
%\end{equation}
%and, with $\rho$ as in Lemma~\ref{main_lemma}, a positive integer $N$ such that
%\begin{equation}\label{N_choice}
%\frac{\rho^{N}}{1-\rho} < \frac{\gamma_{2}-2}{2}.
%\end{equation}
%Given tuples $\mathbf{y} = (y_{1}, \ldots, y_{j+1})$ and $\mathbf{z} = (z_{1}, \ldots, z_{j+1})$, we say that $\mathbf{y} \geqslant \mathbf{z}$ if $y_{i} \geqslant z_{i}$ for all $i = 1, \ldots, j+1$. We call a function $f: \mathbb{N}_{0}^{j+1} \rightarrow \mathbb{R}$ monotonically increasing if $f(\mathbf{y}) \geqslant f(\mathbf{z})$ whenever $\mathbf{y} \geqslant \mathbf{z}$. 
We choose $\varepsilon > 0$ and $N \in \mathbb{N}$ such that 
\begin{equation}
\sum_{s=0}^{N} \Big(\frac{1}{2} + \varepsilon\Big)^{s} + \sum_{s=N+1}^{\infty} \rho^{s} < \gamma_{2}. \nonumber
\end{equation}
We now set $M = M_{\varepsilon}$ as in Lemma~\ref{main_lemma} and $K = K_{M,N}$ as in \eqref{K_{M,N}}. Define an i.i.d.\ sequence $\{Y_{i}\}_{i}$, independent of both the sequences $\{x_{i}\}_{i}$ and $\{b_{i}\}_{i}$, whose distribution is
\[ \Prob[Y_{i} > s] =
\begin{cases}
\Big(\frac{1}{2}+\varepsilon\Big)^{s} & \text{for all } s = 1, \ldots, N,\\
\min\Big\{\Big(\frac{1}{2} + \varepsilon\Big)^{N}, \rho^{s}\Big\} & \text{for all } s > N.
\end{cases} \]
By our choice of $\varepsilon$ and $N$, we have $\Ex[Y_{i}] < \gamma_{2}$, so that 
\begin{equation}\label{all_Y_{j}_smaller}
\inf_{j \in \mathbb{N}} \Prob\Big[\sum_{i=1}^{\ell} Y_{i} < \gamma_{2} \ell \text{ for all } \ell = 1, \ldots, j\Big] > 0.
\end{equation} 
Define $U_{j+1} = (b_{K+j+1} - b_{K+j}) \1_{b_{K+\ell} \leqslant b_{K}+\gamma_{2}\ell \text{ for all } \ell = 1, \ldots, j}$, for all $j \geqslant 0$. By Lemma~\ref{main_lemma}, conditioned on $\mathcal{G}_{K+j}$, the random variable $U_{j+1}$ is stochastically dominated by $Y_{j+1}$, for all $j \geqslant 0$. We claim that
\begin{equation}\label{stoch_dom_claim}
(U_{1}, \ldots, U_{j}) \preccurlyeq (Y_{1}, \ldots, Y_{j}) \text{ for all } j \in \mathbb{N},
\end{equation}
where ``$\preccurlyeq$'' denotes the usual stochastic domination of random vectors. We prove this claim by induction on $j$. Let $f$ be a test function that is monotonically increasing and bounded. Note that $Y_{j+1}$ is independent of $\mathcal{G}_{K+j}$. We then have
\begin{align}
\Ex f(U_{1}, \ldots, U_{j}, U_{j+1}) &= \Ex \Ex\big[f(U_{1}, \ldots, U_{j}, U_{j+1})\big|\mathcal{G}_{K+j}\big] \nonumber\\
&\leqslant \Ex \Ex\big[f(U_{1}, \ldots, U_{j}, Y_{j+1})\big|\mathcal{G}_{K+j}\big] \nonumber\\
&= \Ex f(U_{1}, \ldots, U_{j}, Y_{j+1}) \nonumber\\
&= \Ex \Ex\big[f(U_{1}, \ldots, U_{j}, Y_{j+1})\big|Y_{j+1}\big] \nonumber\\
&\leqslant \Ex \Ex\big[f(Y_{1}, \ldots, Y_{j}, Y_{j+1})\big|Y_{j+1}\big] \nonumber\\
&= \Ex f(Y_{1}, \ldots, Y_{j}, Y_{j+1}). \nonumber
\end{align}
The first of the two inequalities above follows from the fact that $U_{1}, \ldots, U_{j}$ are measurable with respect to $\mathcal{G}_{K+j}$, and conditioned on $\mathcal{G}_{K+j}$, we have $U_{j+1} \preccurlyeq Y_{j+1}$. The second inequality follows from induction hypothesis and the fact that both $(U_{1}, \ldots, U_{j})$ and $(Y_{1}, \ldots, Y_{j})$ are independent of $Y_{j+1}$. This proves the claim of \eqref{stoch_dom_claim} by induction. 

To conclude, we define the transformation $\Psi$ on sequences $\mathbf{y} = (y_{\ell})_{\ell = 1, \ldots, j}$ by $(\Psi \mathbf{y})_{1} = y_{1}$ and for all $m = 2, \ldots, j+1$, 
\[
 (\Psi \mathbf{y})_{m} = 
 \begin{cases} 
 y_{m} & \text{if } \sum_{\ell=1}^{r} y_{\ell} \leqslant \gamma_{2}r \text{ for all } r = 1, \ldots, m-1; \\
 \infty & \text{otherwise.}
 \end{cases}
\]
Since $\Psi$ is non-decreasing, $\Psi(U_{1}, \ldots, U_{j}) \preccurlyeq \Psi(Y_{1}, \ldots, Y_{j})$, and by \eqref{all_Y_{j}_smaller}, we get $\Prob[\mathcal{B}_{\infty}^{M,N}] > 0$, concluding the proof of Lemma~\ref{main_2}. 

\subsection{Reversed trajectories and $h$-transform}\label{proof:subsec:1}

Consider the trajectory of the $k$-th explorer, denoted $\{X^{k}_{n}\}_{0 \leqslant n \leqslant \tau_{k}}$, where $\tau_{k}$ is the first time that it hits the barrier at $b_{k-1}$. This path is distributed as a random walk $\widetilde{W}$ coming from $+\infty$ towards $-\infty$, with transition probabilities $p(-1) = 1-p(+1) = q$, between the \emph{first} time it hits $x_{k}$ and the \emph{first} time it hits $b_{k-1}$. The time-reversed walk $W$ goes from $-\infty$ to $\infty$ with jump probabilities $\tilde{p}(+1) = 1-\tilde{p}(-1) = q$. Then $\{X^{k}_{\tau_{k}-n}\}_{0 \leqslant n \leqslant \tau_{k}}$ becomes the piece of $W$ between the \emph{last} time it visits $b_{k-1}$ and the \emph{last} time it visits $x_{k}$. However, the path of $W$ observed after the \emph{last} visit to the origin $x = 0$ has the distribution of an $h$-transformed random walk that is conditioned to start at $0$ and remain strictly positive thereafter. We denote this $h$-transformed walk by $\{Z^{k}_{n}\}_{n \geqslant 0}$, and let $\sigma^{k}_{x}$ denote the last time it visits the site $x$, for all $x \in \mathbb{Z}^{+}$. We can thus identify $\{X^{k}_{\tau_{k}-n}\}_{0 \leqslant n \leqslant \tau_{k}}$ with $\big\{Z^{k}_{\sigma^{k}_{b_{k-1}}+n}\big\}_{0 \leqslant n \leqslant \sigma^{k}_{x_{k}} - \sigma^{k}_{b_{k-1}}}$. 

Henceforth, we shall use this representation of the trajectory of the $k$-th explorer. Note that the law of the process $\{Z^{k}_{n+j}\}_{n \geqslant 0}$, conditioned on $\sigma^{k}_{x} = j$, is the same as the law of the process $\{x+Z^{k}_{n}\}_{n \geqslant 0}$, for every site $x \in \mathbb{Z}^{+}$ and every $j \in \mathbb{N}$. This shows that, for each fixed $k$, the processes $\{Z^{k}_{n}\}_{\sigma^{k}_{x} \leqslant n \leqslant \sigma^{k}_{x+1}}$ are i.i.d.\ over all sites $x \geqslant 0$.

Note, from \eqref{c_{k}_defn}, that the random variable $c_{k+1} - b_{k}$, conditioned on $\mathcal{G}_{k}$, has the same distribution as the site on the positive integer line where the random walk $\{Z^{k+1}_{n}\}_{n}$ jumps to the left for the first time. Equivalently,
\begin{equation}\label{c_{k+1}-b_{k}_reversed_walk}
\Prob[c_{k+1} - b_{k} > s\big|\mathcal{G}_{k}] = \Prob\big[Z^{k+1}_{n} = n \text{ for all } n = 1, \ldots s+1\big]. 
\end{equation}

\subsection{Parity of the sum of i.i.d.\ variables}\label{proof:subsec:2} 
We make an elementary observation in this subsection.
\begin{lemma}\label{lem:parity_general}
Let $\{X_{n}\}$ be an i.i.d.\ sequence of integer-valued random variables. If $0 < \Prob[X_{1} \equiv 0 \bmod 2] < 1$, then 
\begin{equation}
\lim_{n \rightarrow \infty} \Prob\Big[\sum_{i=1}^{n} X_{i} \equiv 0 \bmod 2\Big] = \frac{1}{2}. \nonumber
\end{equation}
\end{lemma}

\begin{proof}
We define the random variable $Y_{i}$ which takes the value $1$ when $X_{i}$ is even, and the value $-1$ when $X_{i}$ is odd. Note that $\sum_{i=1}^{n} X_{i}$ is even if and only if an even number of the $X_{i}$'s are odd, which is equivalent to $\prod_{i=1}^{n} Y_{i}$ being equal to $1$. Note that $p_{1} = \Prob[Y_{1} = 1]$, and let $p_{n} = \Prob[\prod_{i=1}^{n} Y_{i} = 1]$ for all $n \geqslant 2$. Then $\Ex[\prod_{i=1}^{n} Y_{i}] = 2p_{n}-1$. Since we assume that $0 < p_{1} < 1$, we have $-1 < 2p_{1}-1 < 1$. Using independence, we have
\begin{equation}
|2p_{n+1}-1| = \big|\Ex[\prod_{i=1}^{n+1} Y_{i}]\big| = |2p_{n}-1| |2p_{1}-1|, \nonumber
\end{equation}
implying
\begin{equation}%\label{rate_converge}
\big|p_{n} - \frac{1}{2}\big| = \frac{1}{2}|2p_{1}-1|^{n}, \nonumber
\end{equation}
which establishes exponential decay. This shows that $p_{n}$ converges to $\frac{1}{2}$ at an exponentially fast rate that depends only on $p_{1}$.
\end{proof} 

%We start with the following remark that is crucial for the proof. 
%For each $1 \leqslant s \leqslant N$, we consider the number of visits to the site $b_{k}+s$ by $\{Z^{j}_{n}\}_{\sigma^{j}_{b_{k}} \leqslant n \leqslant \sigma^{j}_{x_{j}}}$ for all $j = k-M, \ldots, k+1$. Given $x_{1}, \ldots, x_{k}$ and $b_{k}$, on the event $\mathcal{B}_{k}^{M,N}$, we have $x_{j} > b_{k}+N$ for all $j = k-M, \ldots, k+1$.

\subsection{The conditional distribution of $a_{k+1} - b_{k}$}\label{proof:subsec:3} Fix $M \geqslant 2$ and $N \geqslant 2$, set $K = K_{M,N}$ as in \eqref{K_{M,N}}, and fix $k \geqslant K_{M,N}$. Later on we shall select $M$ depending on $\varepsilon$ whereas $N$ will remain free. The definitions and estimates that follow depend implicitly on $M$ and $N$. In this subsection, we prove \eqref{b_{k+1}-b_{k}_tail_upto_N}. Recalling from \S\ref{proof:subsec:1} how $\{Z^{j}_{n}\}_{n}$ relates with $\{X^{j}_{n}\}_{n}$, we introduce 
\begin{equation}
D_{k,s} = \sum_{j=1}^{k} \sum_{n = \sigma^{j}_{b_{j-1}}}^{\sigma^{j}_{x_{j}}} \1_{Z^{j}_{n}=b_{k}+s} \nonumber
\end{equation}
for all $s \geqslant 1$, and rewrite \eqref{a_{k}_defn} as follows:
\begin{equation}\label{new_representation_a}
a_{k+1} = b_{k} + \min\{s \geqslant 1: D_{k,s} \equiv 1 \bmod 2\}. \nonumber
\end{equation} 

In order to analyse $D_{k,s}$, we break it into a ``known'' component and an ``unknown'' component conveniently. To this end, for $s \geqslant 0$, define the $\sigma$-field $\mathcal{S}_{k}^{s}$ which contains the following information: 
\begin{enumerate}
\item the values of $x_{j}$ for all $j \in \mathbb{N}$;
\item the value of $b_{k}$;
\item the trajectories $\{Z^{j}_{n}\}_{n = \sigma^{j}_{b_{j-1}}, \dots, \sigma^{j}_{x_{j}}}$ for $j = 1, \ldots, k-M-1$;
\item\label{iv} the pieces $\{Z^{j}_{n}\}_{n = \sigma^{j}_{b_{j-1}}, \dots, \sigma^{j}_{b_{k}+s}}$ for $j = k-M, \ldots, k+1$.
\end{enumerate}
Note that each $\mathcal{S}_{k}^{s}$ is a finer $\sigma$-field than $\mathcal{G}_{k}$. 

For every $s \geqslant N$, we write $D_{k,s} = D^{1}_{k,s} + D^{2}_{k,s}$ where
\begin{equation}
D_{k,s}^{1} = \sum_{j = 1}^{k-M-1} \sum_{n=\sigma^{j}_{b_{j-1}}}^{\sigma^{j}_{x_{j}}} \1_{Z^{j}_{n} = b_{k}+s} + \sum_{j=k-M}^{k+1} \sum_{n=\sigma^{j}_{b_{j-1}}}^{\sigma^{j}_{b_{k}+s-1}} \1_{Z^{j}_{n} = b_{k}+s} \nonumber
\end{equation}
and
\begin{equation}\label{random_part_D^{2}}
D_{k,s}^{2} = \sum_{j=k-M}^{k+1} \sum_{n=\sigma^{j}_{b_{k}+s-1}}^{\sigma^{j}_{b_{k}+s}} \1_{Z^{j}_{n} = b_{k}+s}.
\end{equation} 
Note that $D_{k,s}^{1}$ is measurable with respect to $\mathcal{S}_{k}^{s-1}$ and $D_{k,s}^{2}$ is independent of $\mathcal{S}_{k}^{s-1}$. This in turn implies that the event $\{a_{k+1} - b_{k} > s\}$ is contained in $\mathcal{S}_{k}^{s}$. We now show that, for every $s \geqslant 1$, on the event $\{a_{k+1} - b_{k} > s\}$, we have
\begin{equation}\label{ind_claim_new}
\Prob[a_{k+1} - b_{k} > s\big|\mathcal{S}_{k}^{s-1}] \leqslant \frac{1}{2} + \varepsilon.
\end{equation}
From this we establish \eqref{b_{k+1}-b_{k}_tail_upto_N} inductively on $s$, as follows:
\begin{align}
\Prob[a_{k+1}-b_{k} > s+1\big|\mathcal{G}_{k}] &= \Ex[\1_{a_{k+1} - b_{k} > s+1}\big|\mathcal{G}_{k}] \nonumber\\
&= \Ex\big[\Ex[\1_{a_{k+1} - b_{k} > s+1}\big|\mathcal{S}_{k}^{s}]\big|\mathcal{G}_{k}\big] \nonumber\\
&= \Ex\big[\Ex[\1_{a_{k+1} - b_{k} > s} \1_{a_{k+1}-b_{k} > s+1}\big|\mathcal{S}_{k}^{s}]\big|\mathcal{G}_{k}\big] \nonumber\\
&= \Ex\big[\1_{a_{k+1} - b_{k} > s} \Ex[\1_{a_{k+1}-b_{k} > s+1}\big|\mathcal{S}_{k}^{s}]\big|\mathcal{G}_{k}\big] \nonumber\\
&\leqslant \Big(\frac{1}{2} + \varepsilon\Big) \Ex[\1_{a_{k+1} - b_{k} > s}\big|\mathcal{G}_{k}] \leqslant \Big(\frac{1}{2} + \varepsilon\Big)^{s+1}. \nonumber 
\end{align}

The proof of \eqref{ind_claim_new} happens via induction on $s$ as well. The base case of $s = 1$ is resolved by noting that, conditioned on $\mathcal{S}_{k}^{0}$, the random variable $D_{k,1}^{2}$ is the sum of i.i.d.\ random variables $\sum_{n=\sigma^{j}_{b_{k}}}^{\sigma^{j}_{b_{k}+1}} \1_{Z^{j}_{n} = b_{k}+1}$ for $j = k-M, \ldots, k+1$, and applying Lemma~\ref{lem:parity_general} for all $M$ large enough. 

Suppose we have proved \eqref{ind_claim_new} for some $s < N$. The probability $\Prob[a_{k+1} - b_{k} > s+1\big|\mathcal{S}_{k}^{s}]$ is equal to
\begin{equation}
\Prob\Big[\sum_{j=k-M}^{k+1} \sum_{n = \sigma^{j}_{b_{k}+s}}^{\sigma^{j}_{b_{k}+s+1}} \1_{Z^{j}_{n} = b_{k}+s+1} + D_{k,s+1}^{1} \equiv 0 \bmod 2\big|\mathcal{S}_{k}^{s}\Big], \nonumber 
\end{equation}
where the summands $\sum_{n = \sigma^{j}_{b_{k}+s}}^{\sigma^{j}_{b_{k}+s+1}} \1_{Z^{j}_{n} = b_{k}+s+1}$ are i.i.d.\ over $j = k-M, \ldots, k+1$. By Lemma~\ref{lem:parity_general}, the above probability is bounded above by $\frac{1}{2} + \varepsilon$ if $M$ is chosen large enough. This completes the inductive proof.

\subsection{The conditional distribution of $c_{k+1} - b_{k}$}\label{proof:subsec:4} Here, $N \geqslant 2$ is fixed and \eqref{b_{k+1}-b_{k}_tail_beyond_N} will be established. We fix $k$ and henceforth omit the superscript from $\{Z_{n}^{k+1}\}_{n}$. Recall from \S\ref{model} that $q$ denotes the random walk parameter. The analysis of $c_{k+1} - b_{k}$ will be split into two parts: $q = \frac{1}{2}$ and $\frac{1}{2} < q < 1$. 

\subsubsection{When the sandpile dynamics has no bias}\label{no_bias} We assume that $q = \frac{1}{2}$. The $h$-transformed walk $\{Z_{n}\}_{n}$ uses the harmonic function $h(x) = x$ for $x \in \mathbb{Z}_{+}$. Therefore the transition kernel of $\{Z_{n}\}_{n}$ is given by 
\begin{equation}\label{transition_h_transform_unbiased}
\Prob[Z_{n+1} = x+1|Z_{n} = x] = \frac{1}{2} \big(\frac{x+1}{x}\big) \text{ for all } x \in \mathbb{N}, 
\end{equation}
From \eqref{c_{k+1}-b_{k}_reversed_walk}, for all $s > N$, we have
\begin{align}%\label{c_{k+1}-b_{k}_bound_unbiased}
\Prob[c_{k+1} - b_{k} > s\big|\mathcal{G}_{k}] &= \prod_{x=1}^{s} \Prob[Z_{x+1} = x+1|Z_{x} = x] \nonumber\\
&= \prod_{x=1}^{s} \frac{1}{2} \big(\frac{x+1}{x}\big) = \frac{s+1}{2^{s}}, \nonumber
\end{align}
thus establishing \eqref{b_{k+1}-b_{k}_tail_beyond_N} in this case.

\subsubsection{When the jumps have a negative bias}\label{neg_bias} We assume that $\frac{1}{2} < q < 1$. Let $\Prob_{x}$ denote the law of a random walk $\{W_{n}\}_{n}$ that starts at $W_{0} = x$ and jumps by an increment of $+1$ with probability $q$ and by $-1$ with probability $1-q$. Recall from \S\ref{proof:subsec:1} that $\{Z_{n}\}_{n}$ is the $h$-transform of $\Prob_{0}$, i.e.\ it starts at $0$ and is conditioned on staying strictly positive thereafter. We establish \eqref{b_{k+1}-b_{k}_tail_beyond_N} by providing an explicit form for the transition kernel of $\{Z_{n}\}_{n}$. This is given by
\begin{equation}\label{transition_h_transform_-ve_bias}
\Prob[Z_{n+1} = x+1|Z_{n} = x] = \frac{q^{x+1} - (1-q)^{x+1}}{q^{x} - (1-q)^{x}}
\end{equation}
for all $n$ and all $x \in \mathbb{N}$, and $\Prob[Z_{1} = 1|Z_{0} = 0] = 1$. 

Then, for all $s > N$, from \eqref{c_{k+1}-b_{k}_reversed_walk}, we have
\begin{align}\label{c_{k+1}-b_{k}_bound_+ve_bias}
\Prob[c_{k+1} - b_{k} > s\big|\mathcal{G}_{k}] &= \Prob\big[Z_{n} = n \text{ for all } x = 1, \ldots s+1\big] \nonumber\\
&= \prod_{x=1}^{s} \Prob[Z_{x+1} = x+1|Z_{x} = x] \nonumber\\
&= \prod_{x=1}^{s} \frac{q^{x+1} - (1-q)^{x+1}}{q^{x} - (1-q)^{x}} \nonumber\\
&= \frac{q^{s+1} - (1-q)^{s+1}}{2q - 1}, \nonumber
\end{align}
which establishes \eqref{b_{k+1}-b_{k}_tail_beyond_N}.

It remains to justify \eqref{transition_h_transform_-ve_bias}. Note that the function $h$, defined as
\begin{equation}\label{h_transform_explicit}
h(x) = 1 - \Big(\frac{1-q}{q}\Big)^{x} \text{ for all } x \in \mathbb{Z}_{+}, \nonumber
\end{equation}
is strictly positive and $q$-harmonic on $\mathbb{N}$, and $h(0) = 0$. Hence the law of $\{Z_{n}\}_{n}$ is given by $Z_{1} = 1$ and
\begin{equation}
\Prob[Z_{n+1} = x+1|Z_{n} = x] = \frac{h(x+1)}{h(x)} q \text{ for all } x \geqslant 1, \nonumber
\end{equation} 
giving us \eqref{transition_h_transform_-ve_bias}.

\parskip 0pt
\setstretch{1}
\small
\bibliographystyle{bib/leoabbrv}
\bibliography{bib/leo}

\end{document}